\documentclass{amsart}
\usepackage{cite}
\usepackage[final]{graphicx}
\usepackage[margin=1.5in,letterpaper]{geometry}
\usepackage{booktabs}
\usepackage{url}
\usepackage[T1]{fontenc}
\usepackage{ae,aecompl}
\usepackage[utf8]{inputenc}

\usepackage{caption}
\usepackage{hyphenat}
\usepackage{tikz}
\usetikzlibrary{matrix,arrows}
\tikzset{cdarrow/.style={auto,
    execute at begin node=$\scriptstyle,execute at end node=$}}

\newread\testin

\graphicspath{{draws/}}

\newcommand{\RR}{\mathbb R}

\newcommand{\EE}{\mathbb E}

\newcommand{\abs}[1]{{\lvert #1 \rvert}}

\theoremstyle{plain}
\numberwithin{equation}{section}
\newtheorem{theorem}{Theorem}

\newtheorem{otheorem}[equation]{Theorem}
\newtheorem{lemma}[equation]{Lemma}
\newtheorem{corollary}[equation]{Corollary}

\theoremstyle{definition}
\newtheorem{definition}[equation]{Definition}

\theoremstyle{remark}

\newtheorem{remark}[equation]{Remark}
\newcommand{\Edges}{{\mathcal E}}
\newcommand{\Verts}{{\mathcal V}}

\begin{document}
\title{Measurement Isomorphism of Graphs}

\author[Gortler]{Steven J. Gortler}
\author[Thurston]{Dylan P. Thurston}

\begin{abstract}
The d-measurement set of a graph is its set of
possible squared edge lengths over all d-dimensional
embeddings. 
In this note, we define a new notion of graph isomorphism called
\emph{d-measurement isomorphism}. Two graphs are d-measurement isomorphic
if there is  agreement in their d-measurement sets.
A natural question to ask is ``what can be said about two graphs that
are d-measurement isomorphic?''
In this note, we show that this property coincides with the
\emph{$2$-isomorphism} property studied by Whitney.

\end{abstract}

\date{\today}

%\primaryclass{}
%\secondaryclass{}
%\keywords{}

\maketitle

%\tableofcontents

\section{Introduction}

Given a graph $\Gamma$ we can consider placing each vertex at some
position in $\EE^d$ and then measuring the squared Euclidean 
length of each of the
graph's edges. This gives us the coordinates of a single 
``measurement point'' in 
$\RR^e$, where $e$ is the number of edges in the graph. As we alter
the vertex positions, the measurement point will typically change. The
d-dimensional measurement set, $M_d(\Gamma)$,
is the union of all achievable measurement points 
as we vary over all possible placements of the vertices in $\EE^d$.
Suppose that, 
after some  permutation of the $e$ coordinate axes, 
we have agreement in the d-dimensional measurement sets of 
two graphs,
$\Gamma$ and $\Delta$.
Then we say 
that the graphs $\Gamma$ and $\Delta$ are d-measurement isomorphic.

Clearly, two isomorphic graphs must be d-measurement isomorphic.
But the converse is not true.
For example, 
the measurement set of \emph{any} forest graph is the entire first octant of 
$\RR^e$ as there are no constraints on the achievable edge lengths.
A natural question to ask is ``what can be said about two graphs that
are d-measurement isomorphic?''
In this note, 
we relate this type
of isomorphism to a graph property 
studied by Whitney~\cite{whit}
called $2$-isomorphism.
Our main result is that 
for any $d$, two graphs are d-measurement isomorphic
if and only if they are 2-isomorphic. In particular, 
for 3-connected graphs, this means that two graphs 
are d-measurement isomorphic
if and only if they  are isomorphic graphs.

\label{sec:definitions}
\begin{definition}
  A \emph{graph}~$\Gamma$ is a set of $v$ vertices $\Verts(\Gamma)$
  and $e$ edges~$\Edges(\Gamma)$, where $\Edges(\Gamma)$ is a set of\
  two\hyp element subsets of $\Verts(\Gamma)$.  
\end{definition}

\begin{definition}
Two graphs $\Gamma$ and $\Delta$, are \emph{isomorphic}
if there is a bijection $\varrho$ between $\Verts(\Gamma)$ and $\Verts(\Delta)$
such that $\{x,y\}\in \Edges(\Gamma)$ iff
$\{\varrho(x),\varrho(y)\}\in \Edges(\Delta)$.
\end{definition}

Next we define two  weaker notions of graph isomorphism. These
allow us to move around barely connected parts of the graph without
changing the equivalence class.

\begin{figure}
\begin{center} \begin{tabular}{cc}
  {  \includegraphics[scale=.18]{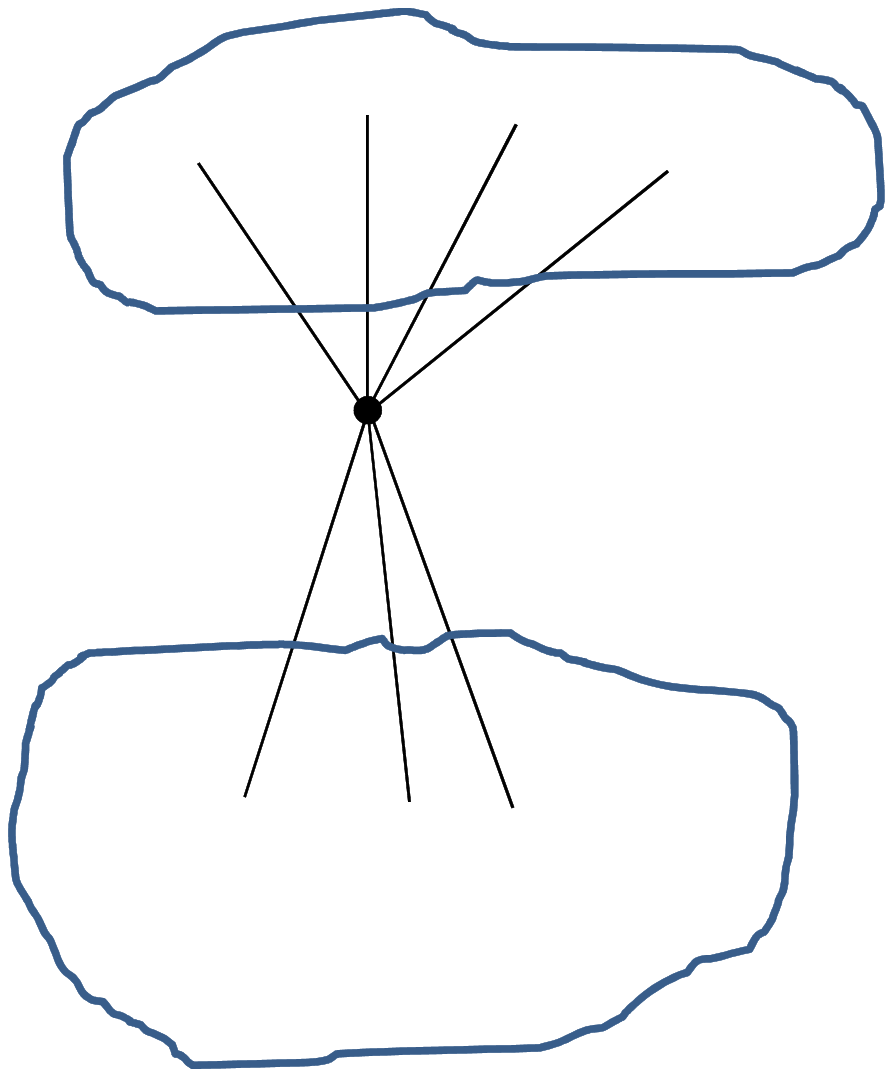}}&
  {\includegraphics[scale=.18]{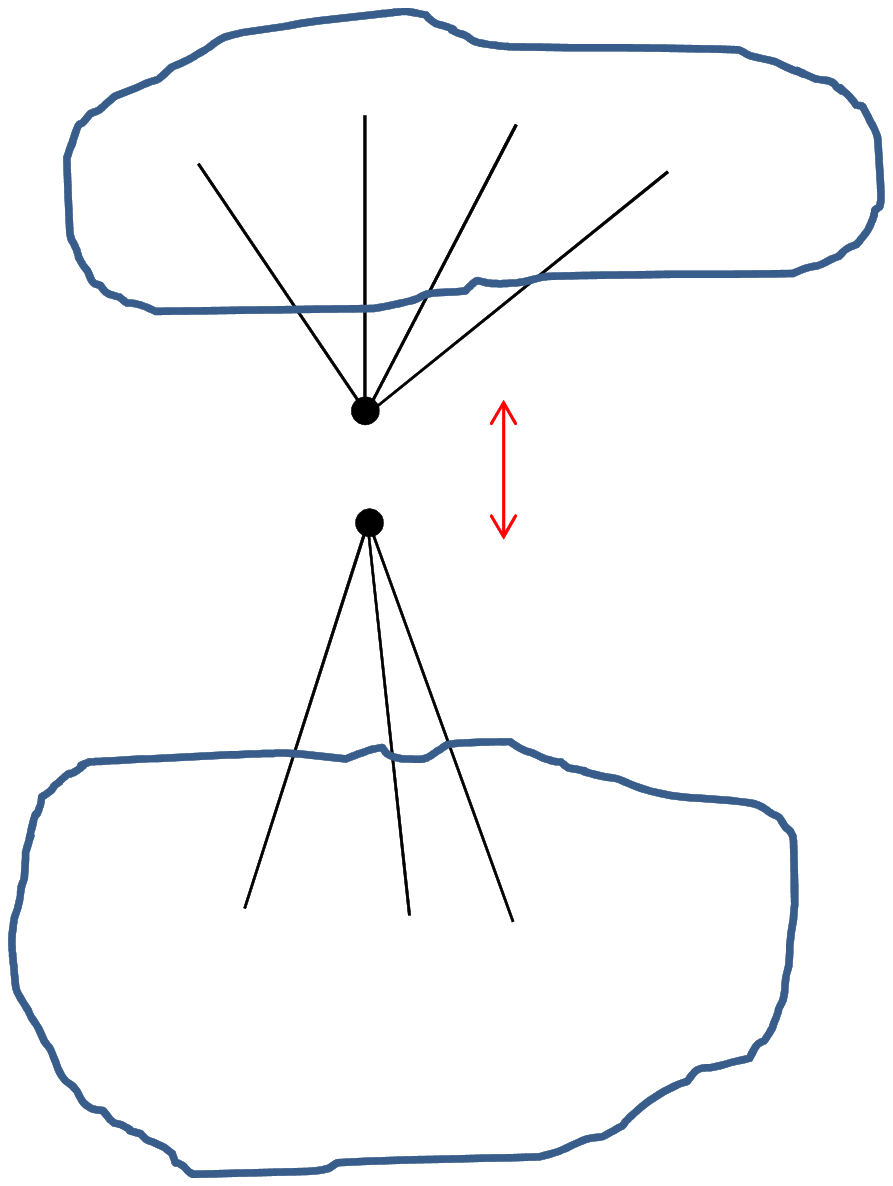}}\\
  (a) & (b) \\
\end{tabular} \end{center}
  \caption{The split operation. The graphs (a) and (b) are 1-isomorphic.}
  \label{fig:split}
\end{figure}

\begin{definition}
A \emph{cut vertex} of  graph is a vertex whose removal disconnects the graph.
A \emph{split} operation breaks a cut vertex into two vertices to produce two
disjoint subgraphs. Two graphs are \emph{1-isomorphic} if they become 
isomorphic under a finite sequence of split operations. 
See Figure~\ref{fig:split}.
\end{definition}

\begin{figure}
\begin{center} \begin{tabular}{cccc}
  {  \includegraphics[scale=.18]{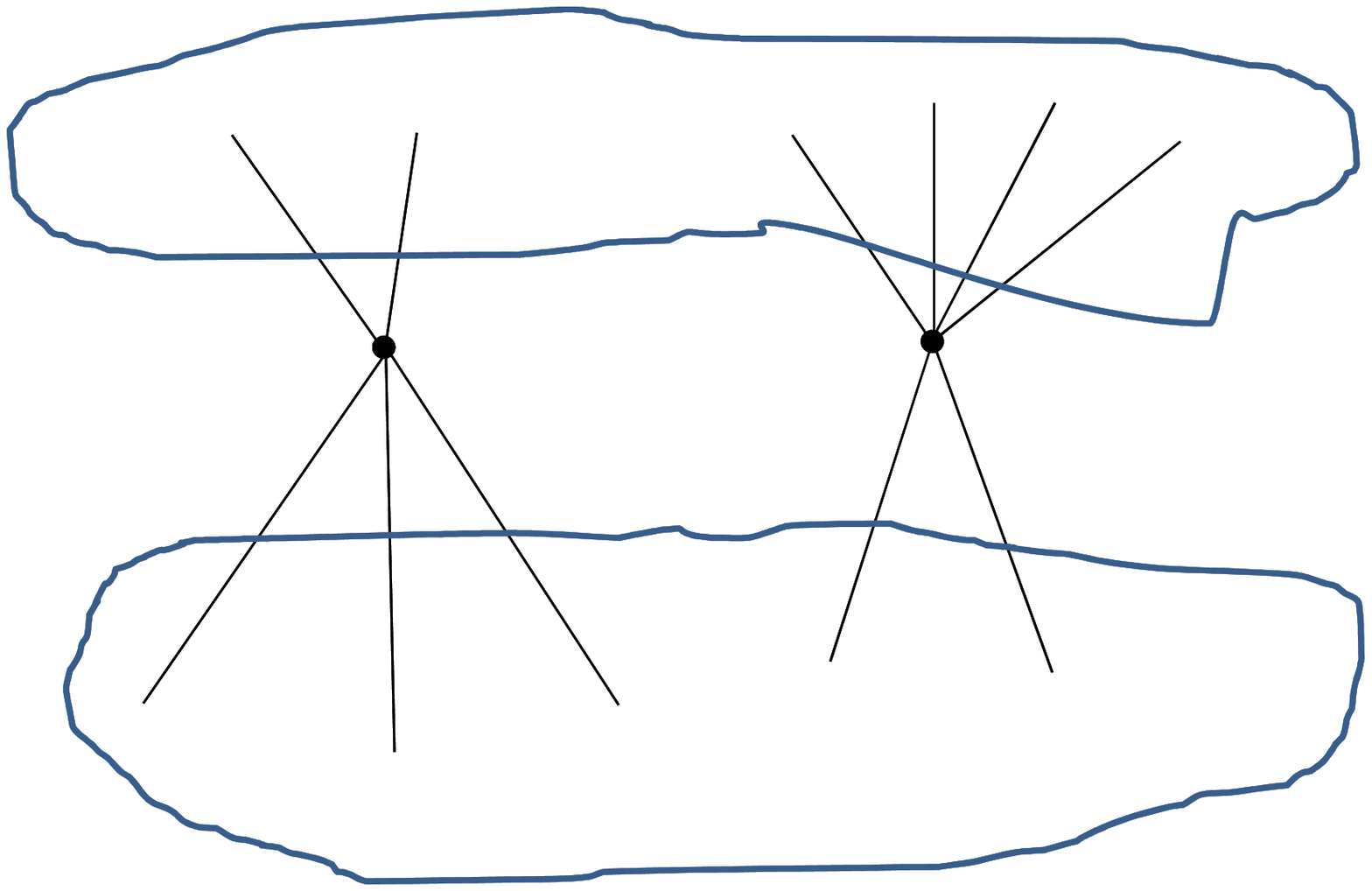}}&
  {  \includegraphics[scale=.18]{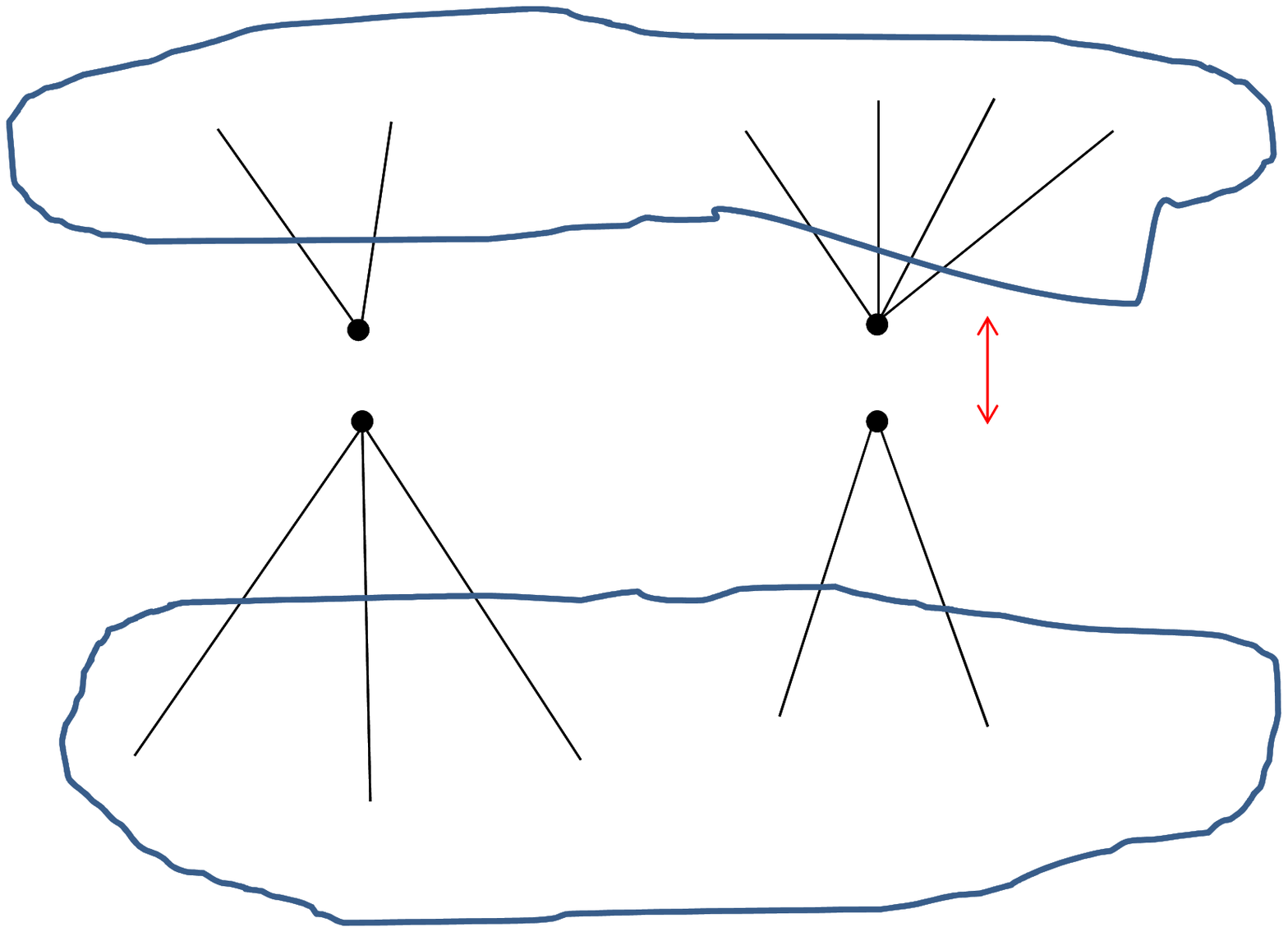}}&
  {  \includegraphics[scale=.18]{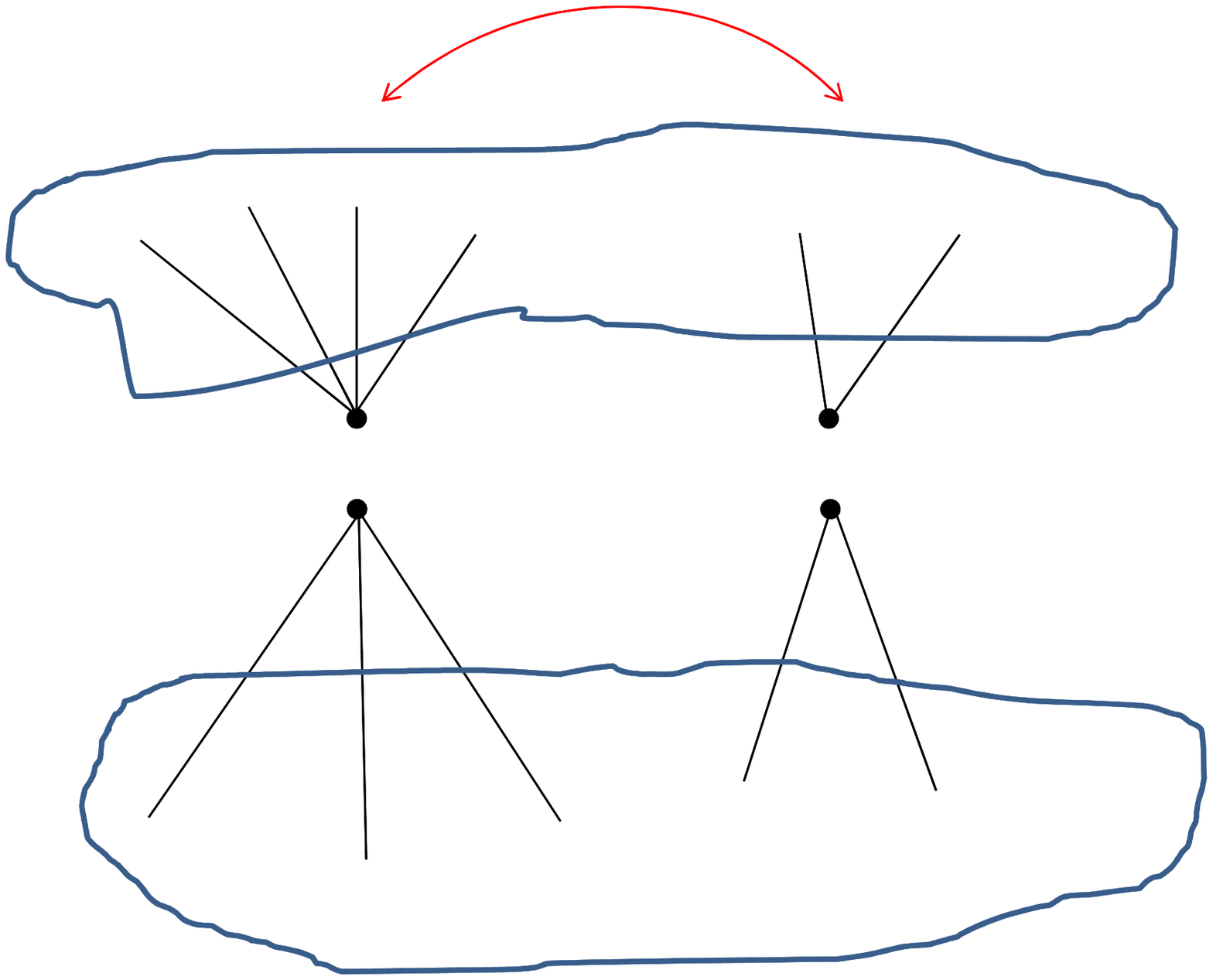}}&
  {  \includegraphics[scale=.18]{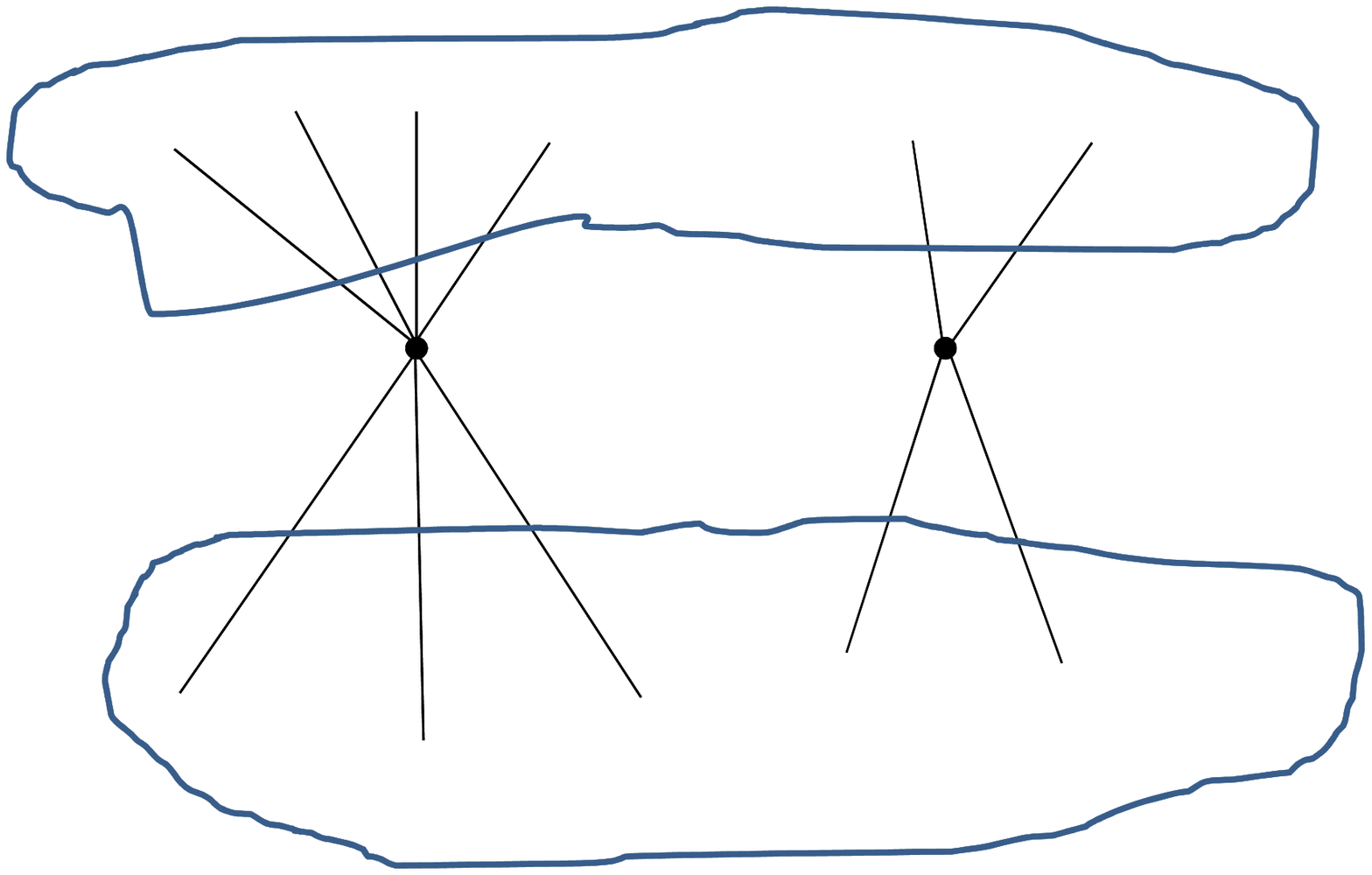}}\\
  (a) & (b) &(c)&(d) \\
\end{tabular} \end{center}
  \caption{The reversal operation. The graphs, (a) and (d) are 2-isomorphic.
Note that the edge lengths of the frameworks are unchanged.}
  \label{fig:reverse}
\end{figure}

\begin{definition}(Following~\cite{wag})
If $S \subset \Edges(\Gamma)$
then let $\Gamma[S]$ denotes the subgraph induced by $S$.
A partition $\{S,T\}$ of $\Edges(\Gamma)$ is a \emph{2-separation} of $\Gamma$ if
$|S|\geq 2\leq |T|$ and 
$|\Verts(\Gamma[S]) 
\cap
\Verts(\Gamma[T])| = 2$. 
Let $\{S,T\}$ be a 2-separation of $\Gamma$ and let
the \emph{cut pair}
$\Verts(\Gamma[S]) 
\cap
\Verts(\Gamma[T])$ be $\{x,y\}$. Let $\Gamma'$ be the graph obtained from $\Gamma$ by
interchanging in $\Gamma[S]$ the incidences of the edges
at $x$ and $y$. Then we say that $\Gamma'$ is obtained from $\Gamma$ by a 
\emph{reversal} operation.  
Two graphs are \emph{2-isomorphic} if they become 1-isomorphic after a finite 
sequence of reversals. See Figure~\ref{fig:reverse}.
\end{definition}
Note that is not the same notion of $2$-isomorphism studied in~\cite{yao}.

\begin{remark}
Since 3-connected graphs have no 2-separations, for such graphs
2-isomorphism coincides with graph isomorphism.
\end{remark}

Next we define another notion of graph equivalence.

\begin{definition}
A \emph{cycle} is a path of adjacent 
vertices that starts and ends at the same vertex,
and with no  vertex repeated in the path.
Two graphs $\Gamma$ and $\Delta$, 
are  \emph{cycle isomorphic}
if there is 
bijection $\sigma$ between $\Edges(\Gamma)$ and
 $\Edges(\Delta)$ 
such that  for any $S \subset \Edges(\Gamma)$, 
$\Gamma[S]$ is a cycle iff $\Delta[\sigma(S)]$ is a cycle.
\end{definition}

In~\cite{whit}, Whitney proved the following theorem 
that will provide all of the heavy lifting that we will need in
this note.
\begin{otheorem}[Whitney]
\label{thm:whit}
Two graphs are cycle-isomorphic iff they are 2-isomorphic.
\end{otheorem}

We now define some notions related to graph embeddings.

\begin{definition}
  A \emph{configuration}~$p$  of a vertex set $\Verts$ is 
  a mapping from $\Verts$ to~$\EE^v$.
Let $C^d(\Verts)$
be the space of configurations of $\Verts$ in $\EE^d$.
For $p\in C(\Verts)$ and $u
  \in \Verts$, 
let $p(u)$ denote the image of $u$ under~$p$.  
A \emph{framework} $(p,\Gamma)$ is the pair of a graph and a configuration
of its vertices.
For a given graph~$\Gamma$
  the \emph{length-squared function}
$\ell_\Gamma$ 
is the function assigning to each
  edge of $\Gamma$ its squared edge length in the framework.  That is,
  the component of $\ell_\Gamma(p)$ in the direction of an
  edge $\{u,w\}$ is $\abs{p(u)-p(w)}^2$.
  Once we fix an 
(arbitrarily) 
identification of each 
edge in $\Edges(\Gamma)$
with an associated coordinate axis in $\RR^e$, 
we can interpret the length-squared  function as being of the type:
  $\ell_\Gamma:C^d(\Verts)\rightarrow\RR^e$.

\end{definition}

\begin{definition}
  The $d$-dimensional \emph{measurement set}~$M_d(\Gamma)$ 
  of a graph~$\Gamma$ is defined to be the image in $\RR^e$ of
  $C^d(\Verts)$ under the map $\ell_\Gamma$.  These are nested by
  $M_d(\Gamma) \subset M_{d+1}(\Gamma)$ and eventually stabilize by
  $M_{v-1}(\Gamma)$.
\end{definition}

In our context of measurement sets of graphs, we  define a new
notion of isomorphism.

\begin{definition}
Two graphs, $\Gamma$ and $\Delta$, both with $e$ edges, 
are  \emph{d-measurement isomorphic}
if there is an 
identification of each 
edge in $\Edges(\Gamma)$
with an associated coordinate axis in $\RR^e$, 
and an
identification of each 
edge in $\Edges(\Delta)$
with an associated coordinate axis in $\RR^e$, 
under which
$M_d(\Gamma)=M_d(\Delta)$
\end{definition}

Our main result is the following
\begin{theorem}
For any $d$, 
two graphs are d-measurement isomorphic iff they are 2-isomorphic.
\end{theorem}

This also gives us the following:

\begin{corollary}
For any two integers $d_1$ and $d_2$, if a pair of graphs are
$d_1$-measurement isomorphic then they are
$d_2$-measurement isomorphic.
\end{corollary}

\begin{remark}
Testing cycle isomorphism of graphs is as 
computationally difficult
as testing graph isomorphism~\cite{rao}, and thus so is testing
2-isomorphism and d-measurement isomorphism.
\end{remark}

\section{Proof}

We will prove our theorem through a cycle of implications.
For these arguments we first fix $d$ as it turns out that our arguments do not
depend on it.

\subsection{2-isomorphism $\Rightarrow$ d-measurement isomorphism}

The graphs $\Gamma$ and $\Delta$ are 2-isomorphic if they become isomorphic
after a finite number of splits and reversals. 
Clearly,  a split
operation does not change $M_d$. 

Likewise, let 
$\{S,T\}$ be a 2-separation of $\Gamma$ with cut pair $\{x,y\}$ and let 
$\Gamma$ and 
$\Gamma'$ be related by the reversal across this pair.
Under reversal, 
there is a canonical bijection between $\Edges(\Gamma)$ and 
$\Edges(\Gamma')$. 
Let us fix our edge-axis identifications to be
consistent with this bijection.

For any 
$p \in C^d(\Verts(\Gamma))$, we can reflect, in $\EE^d$, the positions of the
vertices of $\Verts(\Gamma[S]) \backslash \{x,y\}$
across the hyperplane bisecting the segment $\overline{xy}$
to obtain a new configuration
$p'$.  Under this construction, 
we have $\ell_\Gamma(p) = \ell_{\Gamma'}(p')$. See Figure~\ref{fig:reverse}.
Thus $M_d(\Gamma)=M_d(\Gamma')$ and they must be d-measurement
isomorphic.

\subsection{d-measurement isomorphism $\Rightarrow$ cycle isomorphism}

We start by showing that a cycle graph 
on $k$ edges 
is not d-measurement isomorphic to any other graph.

\begin{lemma}
\label{lem:cir}
Let $c$
be a cycle of $k$ edges, and $b$ be any other graph with $k$ edges.
Then $c$ and $b$ are not d-measurement isomorphic.
\end{lemma}
\begin{proof}

If a graph
is a forest with $k$ edges, there are no constraints on any of the
achievable squared edge lengths, 
and thus  its measurement set  is the entire first octant of $\RR^k$.

If a graph
with $k$ edges,  is not a forest 
then it must have a cycle as a subgraph. 
In this case,  its measurement set cannot be
the entire first octant of $\RR^k$ since
there is no  framework (in any dimension) where all but one of the edges 
of the cycle
has zero length.

Thus, if $c$ is a cycle and $b$ a forest, their measurement sets 
cannot agree
under any edge-axis identifications.

If $c$ is a cycle and 
$b$ is neither a cycle or a forest, then $b$  must have an edge
whose removal does not turn $b$ into a forest. Meanwhile the removal
of any edge turns $c$ into a forest. In terms of measurement sets,
edge removal corresponds to projecting the measurement set
onto coordinates associated with
the appropriate $k-1$ edges. These projections for $b$ and
$c$ cannot
not agree, as one produces the measurement set of a forest and one
produces the measurement set of a non-forest.
Since the projected measurement sets do not agree,
the original measurement set cannot have agreed
as well.
\end{proof}

Suppose that $\Gamma$ is not cycle isomorphic to $\Delta$. Then 
for any edge bijection, $\sigma$, there must be an edge subset
$c$ of $\Gamma$ such that $\Gamma[c]$
is a cycle, 
while $\Delta[\sigma(c)]$ 
is not a cycle.

Let us fix our edge-axis identifications to be
consistent with $\sigma$.
Next, let us project the measurement sets $M_d(\Gamma)$ and $M_d(\Delta)$  down
to the subspace of $\RR^e$
corresponding to the edge set of $c$ and $\sigma(c)$ respectively.
In the case of $\Gamma$ we will
obtain the measurement set of a cycle, while for $\Delta$ 
we will obtain 
the measurement set of a non-cycle. These two measurement sets
cannot be the same  by Lemma~\ref{lem:cir}.
Thus $M_d(\Gamma)$ 
cannot be the same as 
$M_d(\Delta)$ under this edge-axis identification. 

This is true for all edge-axis identifications consistent with 
the bijection $\sigma$.
And, by assumption, this is true for all bijections. Thus it is true
for all edge-axis identifications and 
$\Gamma$ and $\Delta$ cannot be d-measurement isomorphic.

\subsection{cycle isomorphism $\Rightarrow$ 2-isomorphism}
This is simply Theorem~\ref{thm:whit}. And we are done.

\bibliographystyle{acm}
\bibliography{miso}
\end{document}